\documentclass{eptcs-modified}
 \providecommand{\event}{} % Name of the event you are submitting to
\usepackage{amsmath}
\usepackage{amssymb}
\usepackage{amsthm}
\usepackage{amscd}
\usepackage{graphics}
\usepackage{soul}
\usepackage{todonotes}

\usepackage{url}
\usepackage{doi}

\long\def\greybox#1{%
    \newbox\contentbox%
    \newbox\bkgdbox%
    \setbox\contentbox\hbox to \hsize{%
        \vtop{
            \kern\columnsep
            \hbox to \hsize{%
                \kern\columnsep%
                \advance\hsize by -2\columnsep%
                \setlength{\textwidth}{\hsize}%
                \vbox{
                    \parskip=\baselineskip
                    \parindent=0bp
                    #1
                }%
                \kern\columnsep%
            }%
            \kern\columnsep%
        }%
    }%
    \setbox\bkgdbox\vbox{
        \pdfliteral{0.85 0.85 0.85 rg}
        \hrule width  \wd\contentbox %
               height \ht\contentbox %
               depth  \dp\contentbox
        \pdfliteral{0 0 0 rg}
    }%
    \wd\bkgdbox=0bp%
    \vbox{\hbox to \hsize{\box\bkgdbox\box\contentbox}}%
    \vskip\baselineskip%
}

\title{Computably discrete represented spaces}

\author{
Eike Neumann
\institute{Department of Computer Science\\Swansea University, Swansea, UK\\}
\email{e.f.neumann@swansea.ac.uk}
\and
Arno Pauly
\institute{Department of Computer Science\\Swansea University, Swansea, UK\\}
\email{Arno.M.Pauly@gmail.com}
\and
Cécilia Pradic
\institute{Department of Computer Science\\Swansea University, Swansea, UK\\}
\email{c.pradic@swansea.ac.uk}
\and
Manlio Valenti
\institute{Department of Computer Science\\Swansea University, Swansea, UK\\}
\email{manliovalenti@gmail.com}
}

\def\titlerunning{Computably discrete spaces}
\def\authorrunning{E.~Neumann, A.~Pauly, C.~Pradic \& M.~Valenti}

\begin{document}
\theoremstyle{definition}
\newtheorem{theorem}{Theorem}
\newtheorem{definition}[theorem]{Definition}
\newtheorem{problem}[theorem]{Problem}
\newtheorem{assumption}[theorem]{Assumption}
\newtheorem{corollary}[theorem]{Corollary}
\newtheorem{proposition}[theorem]{Proposition}
\newtheorem{lemma}[theorem]{Lemma}
\newtheorem{observation}[theorem]{Observation}
\newtheorem{fact}[theorem]{Fact}
\newtheorem{question}[theorem]{Open Question}
\newtheorem{conjecture}[theorem]{Conjecture}
\newtheorem{example}[theorem]{Example}
\newtheorem{remark}[theorem]{Remark}
\newcommand{\dom}{\operatorname{dom}}
\newcommand{\id}{\textnormal{id}}
\newcommand{\Cantor}{{\{0, 1\}^\mathbb{N}}}
\newcommand{\baire}{{\mathbb{N}^{<\mathbb{N}}}}
\newcommand{\Baire}{{\mathbb{N}^\mathbb{N}}}
\newcommand{\Lev}{\textnormal{Lev}}
\newcommand{\hide}[1]{}
\newcommand{\mto}{\rightrightarrows}
\newcommand{\uint}{{[0, 1]}}
\newcommand{\bft}{\mathrm{BFT}}
\newcommand{\lbft}{\textnormal{Linear-}\mathrm{BFT}}
\newcommand{\pbft}{\textnormal{Poly-}\mathrm{BFT}}
\newcommand{\sbft}{\textnormal{Smooth-}\mathrm{BFT}}
\newcommand{\ivt}{\mathrm{IVT}}
\newcommand{\cc}{\textrm{CC}}
\newcommand{\lpo}{\textrm{LPO}}
\newcommand{\llpo}{\textrm{LLPO}}
\newcommand{\aou}{AoU}
\newcommand{\Ctwo}{C_{\{0, 1\}}}
\newcommand{\name}[1]{\textsc{#1}}
\newcommand{\C}{\textrm{C}}
\newcommand{\CC}{\textrm{CC}}
\newcommand{\UC}{\textrm{UC}}
\newcommand{\ic}[1]{\textrm{C}_{\sharp #1}}
\newcommand{\xc}[1]{\textrm{XC}_{#1}}
\newcommand{\me}{\name{P}.~}
\newcommand{\etal}{et al.~}
\newcommand{\eval}{\operatorname{eval}}
\newcommand{\rank}{\operatorname{rank}}
\newcommand{\Sierp}{Sierpi\'nski }
\newcommand{\isempty}{\operatorname{IsEmpty}}
\newcommand{\spec}{\textrm{Spec}}
\newcommand{\cord}{\textrm{COrd}}
\newcommand{\Cord}{\textrm{\bf COrd}}
\newcommand{\CordM}{\Cord_{\textrm{M}}}
\newcommand{\CordK}{\Cord_{\textrm{K}}}
\newcommand{\CordHL}{\Cord_{\textrm{HL}}}
\newcommand{\leqW}{\leq_{\textrm{W}}}
\newcommand{\leW}{<_{\textrm{W}}}
\newcommand{\equivW}{\equiv_{\textrm{W}}}
\newcommand{\geqW}{\geq_{\textrm{W}}}
\newcommand{\pipeW}{|_{\textrm{W}}}
\newcommand{\nleqW}{\nleq_{\textrm{W}}}
\newcommand{\Det}{\textrm{Det}}
\newcommand{\tktm}{$\mathrm{T}2\kappa\mathrm{TM}$}
\newcommand{\Set}[2]{\left\{#1 \mid #2\right\}}
\newcommand{\N}{\mathbb{N}}
\newcommand{\R}{\mathbb{R}}
\newcommand{\NN}{\N^{\N}}
\renewcommand{\O}{\mathcal{O}}
\newcommand{\V}{\mathcal{V}}
\newcommand{\A}{\mathcal{A}}

\newcommand\tboldsymbol[1]{%
\protect\raisebox{0pt}[0pt][0pt]{%
$\underset{\widetilde{}}{\boldsymbol{#1}}$}\mbox{\hskip 1pt}}

\newcommand{\bolds}{\tboldsymbol{\Sigma}}
\newcommand{\boldp}{\tboldsymbol{\Pi}}
\newcommand{\boldd}{\tboldsymbol{\Delta}}
\newcommand{\boldg}{\tboldsymbol{\Gamma}}

\newcounter{saveenumi}
\newcommand{\seti}{\setcounter{saveenumi}{\value{enumi}}}
\newcommand{\conti}{\setcounter{enumi}{\value{saveenumi}}}

\newcommand\cecilia[1]{}

\newcommand\arno[1]{}

\newcommand\eike[1]{}
\newcommand\manlio[1]{}

\maketitle

\begin{abstract}
In computable topology, a represented space is called computably discrete if its equality predicate is semidecidable. While any such space is classically isomorphic to an initial segment of the natural numbers, the computable-isomorphism types of computably discrete represented spaces exhibit a rich structure. We show that the widely studied class of computably enumerable equivalence relations (ceers) corresponds precisely to the computably Quasi-Polish computably discrete spaces. We employ computably discrete spaces to exhibit several separating examples in computable topology. We construct a computably discrete computably Quasi-Polish space admitting no decidable properties, a computably discrete and computably Hausdorff precomputably Quasi-Polish space admitting no computable injection into the natural numbers, a two-point space which is computably Hausdorff but not computably discrete, and a two-point space which is computably discrete but not computably Hausdorff. We further expand an example due to Weihrauch that separates computably regular spaces from computably normal spaces.
\end{abstract}

% \keywords{MSC-class: 	03E15, 54H05, 03D60, 03F15}

\section{Introduction}

A represented space is computably discrete if equality is semidecidable. Up to homeomorphism,  this notion is not particularly interesting: since every computably discrete represented space is countable and classically discrete, computably discrete spaces are classically isomorphic to the natural numbers or a finite initial segment thereof. However, in the realm of computable topology, we can provide a much more fine-grained analysis showcasing how, in fact, computably discrete spaces exhibit a rich structure and can be used as a source of counterexamples to separate the computable counterparts of classical topological notions.

%In this paper, we explore the complex landscape of computably discrete spaces, with a special focus on their relation with computably Quasi-Polish spaces.

After a short introduction on the notation and the main relevant notions (Section~\ref{sec:background}), in Section~\ref{sec:compquasipolish}, we focus on computably discrete computably Quasi-Polish spaces. In particular, we show that such spaces are isomorphic to quotients of $\mathbb{N}$ by a computably enumerable equivalence relation (ceer). The theory of ceers has received a lot of attention over the past few years. In particular, there is an extensive literature on the structure of ceers under so-called \emph{computable reducibility}, a notion of reducibility between ceers that can be seen as a computable counterpart of Borel reducibility on equivalence relations, widely studied in descriptive set theory. In particular, given two ceers $R$ and $S$, the computable reducibility of $R$ to $S$ can be restated as the existence of a computable injection between the quotients $\mathbb{N}/R$ and $\mathbb{N}/S$. For a more thorough overview of the theory of ceers, we refer the reader to \cite{GG2001ceers,ABS2017ceers}.
It is easy to see that computably discrete spaces are not necessarily computably Hausdorff, \textit{i.e.}, computable discreteness does not imply that equality is decidable. We significantly improve this result by constructing an infinite computably discrete computably Quasi-Polish space admitting no non-trivial decidable properties at all, answering a question by Emmanuel Rauzy.

In Section~\ref{sec:discrete_precompQP}, we turn our attention to computably discrete precomputably (\textit{i.e.}, not necessarily computably overt) Quasi-Polish spaces. We show that there is a computably discrete and computably Hausdorff precomputably Quasi-Polish space admitting no computable injection into the natural numbers. We also explore some computational properties of an example employed by Weihrauch to separate computably regular spaces from computably normal spaces.

In Section~\ref{sec:finite_spaces}, we focus on finite spaces. In particular, we show that computably Hausdorff and computably discrete are incomparable notions by building a two-point space that is computably Hausdorff but not computably discrete and a two-point space that is computably discrete but not computably Hausdorff.

Finally, in Section~\ref{sec:N_isomorphism}, we conclude with some observations on
the computable-isomorphism types of $\mathbb{N}$.

This is a slightly extended version of the conference paper \cite{discreteness-cie}.

\section{Background}
\label{sec:background}
We assume some familiarity with the theory of represented spaces as presented e.g.~in \cite{pauly-synthetic}.
To make the paper more self-contained, we provide a quick summary of the main definitions and results.
For the notion of a (pre)computably Quasi-Polish space we refer to \cite{paulydebrecht4}.
\eike{We move from ``general'' to ``concrete''}

\subsection{Represented Spaces}
\manlio{I believe Arno's notation wants the set to be $X$, and the represented space to be $\mathbf{X}=(X,\delta_{\mathbf{X}})$. I may be wrong though, so I'll leave it to Arno to complain (if he wants to).}
\eike{Remaining todo: implement this consistently everywhere}
A represented space is a set $X$ together with a partial surjection $\delta_\mathbf{X} \colon \subseteq \NN \to X$ called the \emph{representation}.
We write $\mathbf{X} = (X, \delta_\mathbf{X})$ for a represented space and its representation.
A point $p \in \NN$ with $\delta_\mathbf{X}(p) = x$ is called a $\delta_\mathbf{X}$-\emph{name} of $x$, or simply a name of $x$.
A point $x \in \mathbf{X}$ in a represented space is called \emph{computable} if it has a computable name.
In the sequel, we will refer to represented spaces simply as ``spaces''.

A \emph{(partial) multi-valued map} $F \colon \subseteq \mathbf{X} \rightrightarrows \mathbf{Y}$ between spaces is simply a relation $F \subseteq \mathbf{X} \times \mathbf{Y}$.
For a multi-valued map $F$, we let $F(x) = \Set{y \in \mathbf{Y}}{(x,y) \in F}$ be the set of its \emph{values} in $x$ and we let
$\dom(F) = \Set{x \in \mathbf{X}}{F(x) \neq \emptyset}$ be its \emph{domain}.
A \emph{realiser} of a multi-valued map
$F \colon \subseteq \mathbf{X} \rightrightarrows \mathbf{Y}$
is a partial map
$R_F \colon \subseteq \NN \to \NN$
with
$\dom(R_F) \supseteq \delta_\mathbf{X}^{-1}(\dom(F))$
and $\delta_{\mathbf{Y}}\left(R_F(p)\right) \in F(\delta_{\mathbf{X}}(p))$ for all $p \in \delta_\mathbf{X}^{-1}(\dom(F))$.
Multi-valued maps differ from relations in how their composition is defined.
Let $F \colon \mathbf{X} \mto \mathbf{Y}$ and $G \colon \mathbf{Y} \mto \mathbf{Z}$ be multi-valued maps.
Define
$\dom(G \circ F) = \Set{x \in \mathbf{X}}{x \in \dom(F) \land F(x) \subseteq \dom(G)}$
and $G \circ F(x) = \bigcup_{y \in F(x)}G(y)$ for all $x \in \dom(G \circ F)$.
Thus, multi-valued maps are composed like relations, but a point $x$ belongs to the domain of $G \circ F$ if and only if
\emph{every} value $y \in F(x)$ belongs to the domain of $G$.
This ensures that if $R_F$ is a realiser of $F$ and $R_G$ is a realiser of $G$, then $R_G \circ R_F$ is a realiser of $G \circ F$.

A multi-valued map $F$ is called \emph{computable} if it has a computable realiser.
To express that $F$ is computable, we will also say that ``we can compute $F$'' or, to emphasize multi-valuedness, ``we can non-deterministically compute $F$''.
A map is called \emph{continuously realisable} or simply \emph{continuous} if it has a continuous realiser.
Note that ``continuity'' in this sense is a priori not connected to any kind of topological continuity.

Partial continuous functions of type $\NN \to \NN$ can be coded by Baire space elements $p \in \NN$, where $p(0)$ is interpreted as the index of a Turing machine and the function $n \mapsto p(n + 1)$ is interpreted as an oracle that the machine has access to.
For represented spaces $\mathbf{X}$ and $\mathbf{Y}$ we define the \emph{function space} $\mathbf{Y}^{\mathbf{X}}$ whose underlying set is the set of all total single-valued continuously realisable functions from $\mathbf{X}$ to $\mathbf{Y}$.
The representation $\delta \colon \subseteq \NN \to \mathbf{Y}^{\mathbf{X}}$ sends a $p \in \NN$ that codes the realiser of a function $f \colon \mathbf{X} \to \mathbf{Y}$ to the function it realises.

The \emph{product} $\mathbf{X} \times \mathbf{Y}$ of represented spaces $\mathbf{X}$ and $\mathbf{Y}$ has as underlying set the Cartesian product of $\mathbf{X}$ and $\mathbf{Y}$.
A sequence $p \in \NN$ is a name of $(x,y) \in \mathbf{X} \times \mathbf{Y}$ if and only if $n \mapsto p(2n)$ is a name of $x$ and $n \mapsto p(2n + 1)$ is a name of $y$.

It can be shown that the category of represented spaces with computable (total single-valued) maps as morphisms is Cartesian closed, with products $\mathbf{X} \times \mathbf{Y}$ and exponentials $\mathbf{Y}^{\mathbf{X}}$ defined as above.

The analogues of topological concepts over represented spaces are introduced based on the notion of ``continuous map'' given above.
Sierpinski space $\mathbb{S}$ is the represented space consisting of the set $\{\top,\bot\}$
and the representation $\delta \colon \NN \to \mathbb{S}$ with $\delta(0^\omega) = \bot$ and $\delta(p) = \top$ for all $p \neq 0^\omega$.

Let $\mathbf{X}$ be a represented space.
A subset $U$ of $\mathbf{X}$ is called \emph{open} if the characteristic function
\[
    \chi_U \colon \mathbf{X} \to \mathbb{S} \qquad\qquad
    \chi_U(x) =
    \begin{cases}
        \top &\text{if }x \in U,\\
        \bot &\text{otherwise.}
    \end{cases}
\]
is continuous.
Dually, a subset $A$ of $\mathbf{X}$ is called \emph{closed} if $\mathbf{X} \setminus U$ is open.
A set is called \emph{computably} open if the above function is computable, and computably closed if its complement is computably open.

By identifying open sets with their characteristic functions, we obtain the space $\O(\mathbf{X})$ of opens by identification with the exponential $\mathbb{S}^\mathbf{X}$.
We obtain the space $\A(\mathbf{X})$ of closed sets by identifying a closed set with its complement.

For a represented space $\mathbf{X}$, the space $\V(\mathbf{X})$ of \emph{overts} of $\mathbf{X}$ is the space of all closed subsets of $\mathbf{X}$, identified with a subspace of $\O(\O(\mathbf{X}))$
via the map $A \mapsto \Set{U \in \O(\mathbf{X})}{U \cap A \neq \emptyset}$.
A closed subset $A$ of $\mathbf{X}$ is called \emph{computably overt} if it is a computable point in $\V(\mathbf{X})$.
In particular, the space $\mathbf{X}$ is called computably overt if we can semi-decide for a given $U \in \O(\mathbf{X})$ if $U$ is non-empty.

A space is called \emph{computably Hausdorff} if the diagonal $\Delta_\mathbf{X} \subseteq \mathbf{X}\times \mathbf{X}$ is a computably closed subset of $\mathbf{X} \times \mathbf{X}$, or in other words, if inequality of points in $\mathbf{X}$ is semidecidable.
Dually, a space is called \emph{computably discrete} if $\Delta_\mathbf{X}$ is a computably open subset of $\mathbf{X} \times \mathbf{X}$, or in other words, if equality of points in $\mathbf{X}$ is semidecidable.

For all $\mathbf{X}$ we have a canonical computable map $\mathbf{X} \to \O(\O(\mathbf{X}))$ sending a point $x$ to the set $\Set{U \in \O(\mathbf{X})}{x \in U}$ of all point sets containing $x$.
In breach with the usual terminology, we will refer to this set as the \emph{neighbourhood filter} of $x$, although it contains only open neighbourhoods.
If this map admits a continuous partial inverse, then $\mathbf{X}$ is called \emph{admissible}.
If this partial inverse is even computable, then $\mathbf{X}$ is called \emph{computably admissible}.
Computable admissibility can be viewed as an effectivisation of $T_0$ separation.
The latter says that a points are determined by their neighbourhood filters.
The former says that points can be computed from their neighbourhood filters.
Computably admissible spaces are closed under subspaces, products, and exponentials.
In fact, the computably admissible spaces form an exponential ideal in the category of represented spaces.
For admissible spaces there is a connection between continuous realisability and continuity:
if $Y$ is an admissible space and $\mathbf{X}$ is an arbitrary space, then a single-valued map
$f \colon \mathbf{X} \to \mathbf{Y}$
is continuously realisable if and only if it is topologically continuous with respect to the final topologies induced by the representations of $\mathbf{X}$ and $\mathbf{Y}$.
In general, every continuously realisable function is topologically continuous in this sense, but there may be topologically continuous functions which are not continuously realisable.
If $\mathbf{X}$ is admissible, then the open subsets and closed subsets as defined above correspond to the topologically open and closed subsets with respect to the final topology induced by the representation.
However, we should warn the reader that products and subspaces of admissible represented spaces do not correspond to products and subspaces of the associated topological spaces.
Rather, products of admissible represented spaces carry the sequentialisation of the product topology, and subspaces carry the sequentialisation of the subspace topology.
In general, the sequentialisation of a topology can be strictly finer than the topology itself.
This implies for example that a computably Hausdorff space need not be Hausdorff in the classical topological sense (since the definition involves a product space). Rather, the topology of a computably Hausdorff space is only sequentially Hausdorff in the sense that every convergent sequence has a unique limit.
See \cite[Example 6.2]{Franklin1965} for an example separating this notion from Hausdorffness.

\subsection{Effectively Countably-Based Spaces}

\eike{TODO: Countably based spaces are particularly nice to work with. Further, topology and ``synthetic topology'' agree.}

\begin{definition}
A represented space $\mathbf{X}$ is called effectively countably-based, if there is a computable map $B_{\cdot} : \mathbb{N} \to \mathcal{O}(\mathbf{X})$ such that the induced computable map $$A \mapsto \bigcup_{n \in A} B_n : \mathcal{V}(\mathbb{N}) \to \mathcal{O}(\mathbf{X})$$ has a computable multi-valued right-inverse.
In this case, we call $(B_n)_n$ an \emph{effective basis} of $\mathbf{X}$.
\end{definition}

Without assuming admissibility, effectively countably based spaces are not closed under subspaces, in fact we even have:
\begin{proposition}
Every represented space $\mathbf{X}$ occurs as a subspace of an effectively countably-based space $\mathbf{X}'$.
\begin{proof}
Let $\delta_\mathbf{X} : \subseteq \Cantor \to \mathbf{X}$ be the representation of $\mathbf{X}$. We let $\mathbf{X}'$ have the underlying set $X \uplus \{\bot\}$, with representation $\delta$ defined as follows: $\delta(\langle p, q\rangle) = \bot$ if $p$ contains infinitely many $1$s, and $\delta(\langle p, q\rangle) = \delta_\mathbf{X}(q)$ if $p$ contains only finitely many $1$s. Then $\mathbf{X}'$ has the indiscrete topology and is overt, which suffices to make it effectively countably based.
\end{proof}
\end{proposition}
However, if $\mathbf{X}$ is effectively countably based and (computably) admissible, then so is any subspace of $\mathbf{X}$.

\begin{definition}
A space $\mathbf{X}$ admits an effectively fibre-overt representation if there is a computable surjection $\delta : \subseteq \Baire \to \mathbf{X}$ such that $\delta^{-1} : \mathbf{X} \mto \Baire$ is computable and $x \mapsto \overline{\delta^{-1}(x)} : \mathbf{X} \to \mathcal{V}(\Baire)$ is computable.
\end{definition}

\begin{proposition}\label{Proposition: characterisation of ccb spaces}
The following are equivalent for a computably admissible space $\mathbf{X}$:
\begin{enumerate}
\item $\mathbf{X}$ is effectively countably-based.
\item $\mathbf{X}$ computably embeds into $\mathcal{O}(\mathbb{N})$, \textit{i.e.}, there is a computable map $i \colon \mathbf{X} \to \O(\N)$ with a computable inverse $i^{-1} \colon i(\mathbf{X}) \to \mathbf{X}$.
\item $\mathbf{X}$ admits an effectively fibre-overt representation.
\end{enumerate}
\end{proposition}
\begin{proof}
    $(1 \Rightarrow 2)$:
    Assume that $\mathbf{X}$ is effectively countably based with basis $(B_{n})_n$.
    Consider the computable map
    \[
        i \colon \mathbf{X} \to \O(\N),
        \;
        i(x) = \Set{n \in \N}{x \in B_{n}}.
    \]
    We claim that this map is a computable embedding.
    Assume we are given $i(x) \in \O(\N)$.
    We just need to show that we can compute $\Set{U \in \O(\mathbf{X})}{x \in U}$ --
    by computable admissibility, this is enough to compute $x$.
    To prove the claim, assume that we are given $U \in \O(\mathbf{X})$.
    Then, since $\mathbf{X}$ is effectively countably based, we can compute a set $A \in \V(\N)$
    with $U = \bigcup_{n \in A} B_n$.
    It is easy to see that $x \in U$ if and only if $A \cap i(x) \neq \emptyset$,
    so that we can semi-decide if $x \in U$. This proves the claim.

    $(2 \Rightarrow 3)$:
    Now assume that $\mathbf{\mathbf{X}}$ computably embeds into $\O(\N)$ via a computable embedding
    $i \colon \mathbf{\mathbf{X}} \to \O(\N)$.
    The space $\O(\N)$ admits the following effectively fibre-overt total representation:
    \[
        \gamma \colon \NN \to \O(\N) \qquad \qquad
        \gamma(p) = \Set{n \in \N}{\exists k. \; p(k) = n + 1}.
    \]
    Define
    $
        \delta \colon \subseteq \NN \to \mathbf{\mathbf{X}}
    $
    with $\dom(\delta) = \gamma^{-1}(i(\mathbf{X}))$ and
    $\delta(p) = i^{-1}\left(\gamma(p)\right)$.
    Then $\delta^{-1} = \gamma^{-1} \circ i$ is computable
    as a multi-valued map $\mathbf{X} \mto \NN$
    and $\overline{\delta^{-1}(\cdot)} = \overline{\gamma^{-1} \circ i(\cdot)}$
    is computable as a map
    $\overline{\delta^{-1}(\cdot)} \colon \mathbf{X} \to \V(\NN)$.

    $(3 \Rightarrow 1)$:
    Finally, assume that $\mathbf{X}$ admits an effectively fibre-overt representation $\delta \subseteq \NN \to \mathbf{X}$.
    Then the map $\delta_* \colon \O(\NN) \to \O(\mathbf{X})$ which sends $U \in \O(\NN)$ to $\delta(U \cap \dom(\delta))$ is well-defined and continuous.
    Indeed, given $U \in \O(\NN)$ and $x \in \mathbf{X}$ we can compute $\overline{\delta^{-1}(x)} \in \V(\NN)$ and accept $x$ if and only if
    $\overline{\delta^{-1}(x)} \cap U \neq \emptyset$.
    It is straight-forward to check that this computes the map $U \mapsto \delta(U \cap \dom(\delta))$.

    Now, observe that $\NN$ is effectively countably based.
    Fixing a computable surjection $\langle\cdot\rangle\colon\N^* \to \N$, an effective basis is given by
    \[B_{\langle n_1,\dots,n_k\rangle} = \Set{p \in \NN}{p(j) = n_j \text{ for }j = 1,\dots,k}.\]

    This effective basis $(B_n)_n$ yields the computable map $n \mapsto \delta_*\left(B_n\right) \colon \N \to \O(\mathbf{X})$.

    Now, suppose we are given an open set $V \in \O(\mathbf{X})$.
    A name of $V$ is the code $p$ a partial function $F \colon \subseteq \NN \to \NN$ with
    $\dom(F) \supseteq \delta_Y^{-1}\left(\mathbf{X}\right)$ satisfying for all $q \in \delta_Y^{-1}\left(\mathbf{X}\right)$ that
    $\delta_\mathbf{X}(q) \in V$ if and only if there exists some $k$ such that $F(q)(k) \neq 0$.

    Given a code $p$ as above we can compute a new code $p'$ representing the same open set via a total function $G \colon \NN \to \NN$ as follows:
    let $p'(n) = p(n)$ for $n > 0$.
    Let $p'(0)$ be the index of a Turing machine implementing the following algorithm: given $n \in \N$, simulate the Turing machine $p(0)$
    for $n$ steps on the inputs $0, \dots, n$. If the machine halts and outputs a non-zero number, output $1$, otherwise output $0$.

    Thus, we can compute in a multi-valued manner an open set $U \in \O(\NN)$ with
    $\delta_*(U) = V$.
    Using that $\NN$ is computably countably based, we can compute some $A \in \V(\N)$ with
    $U = \bigcup_{n \in A} B_n$,
    yielding
    $V = \delta_*(U) = \bigcup_{n \in A} \delta_*(B_n)$.
    Hence, $\left(\delta_*\left(B_{n}\right)\right)_n$ is an effective basis of $\mathbf{X}$.
\end{proof}

Subspaces of countably based admissible represented spaces carry the subspace topology.
The next proposition effectivises this fact:

\begin{proposition}\label{Proposition: extending open sets in ccb spaces}
If $\mathbf{Y}$ is computably admissible and effectively countably-based and $\mathbf{X} \subseteq \mathbf{Y}$ is a subspace, then the restriction map $U \mapsto (X \cap U) : \mathcal{O}(\mathbf{Y}) \to \mathcal{O}(\mathbf{X})$ has a computable multi-valued right-inverse.
\end{proposition}
\begin{proof}
    We may assume without loss of generality that the representation $\delta_\mathbf{Y}$ of $\mathbf{Y}$ is effectively fibre-overt.
    Assume we are given a name of an open set $V \in \O(\textbf{X})$.
    Then as in the proof of $(3 \Rightarrow 1)$ in Proposition \ref{Proposition: characterisation of ccb spaces}, we can
    compute a total function $G \colon \NN \to \NN$ such that for all $q \in \delta_\mathbf{Y}^{-1}\left(\mathbf{X}\right)$ we have
    $\delta_\mathbf{X}(q) \in V$ if and only if there exists some $k$ such that $G(q)(k) \neq 0$.

    Using computability of $\overline{\delta_\mathbf{Y}^{-1}(y)}$ as an overt set and the definition of overtness,
    we can compute the open set
    \[
        U = \Set{y \in \mathbf{Y}}{ \exists q \in \overline{\delta_\mathbf{Y}^{-1}(y)}. \exists k \in \N. G(q)(k) \neq 0}.
    \]
    This set is computable as an open set since $\overline{\delta_\mathbf{Y}^{-1}(y)}$ is computable as an overt set.
    Observe that by continuity we have
    \[
        U = \Set{y \in \mathbf{Y}}{ \exists q \in \delta_\mathbf{Y}^{-1}(y). \exists k \in \N. G(q)(k) \neq 0}.
    \]
    If $x \in \mathbf{X}$ and $q \in \NN$ with $\delta_\mathbf{Y}(q) = x$ then we have by definition of $G$ that $x \in V$ if and only if $G(q)(k) \neq 0$ for some $k$.
    This shows that $U \cap \mathbf{X} = V$.
\end{proof}

\subsection{Computably Quasi-Polish Spaces}
Quasi-Polish spaces were introduced by de Brecht \cite{debrecht6} as a unifying framework for descriptive set theory over Polish spaces and continuous domains. A topological space is Quasi-Polish if and only if it is countably based and admits a total admissible representation (see \cite[Theorem 49]{debrecht6}).
de Brecht, Pauly, and Schröder \cite{paulydebrecht4} have proposed an effectivisation of Quasi-Polish spaces, closely resembling the ideal presentation of effective domains, based on the following characterisation.
Let $\ll$ be a transitive relation over the natural numbers.
An \emph{ideal} with respect to $\ll$ is a non-empty subset $I$ of $\N$
which is \emph{downwards closed}, \textit{i.e.}, for all $x \in \N$, if there exists $y \in I$ with $x \ll y$, then $x \in I$,
and \emph{upwards directed}, \textit{i.e.}, for all $x, y \in I$ there exists $z \in I$ with $x \ll z$ and $y \ll z$.

The \emph{space of ideals} $\mathbf{I}(\ll)$ is the represented space whose underlying set is the set of ideals of $\ll$.
A point $p \in \NN$ represents an ideal $I$ if and only if $I = \Set{p(k)}{k \in \N}$.
A represented space is Quasi-Polish if and only if it is isomorphic to $\mathbf{I}(\ll)$ for some transitive relation $\ll$.
Further, a topological space is Quasi-Polish if and only if it is isomorphic to $\mathbf{I}(\ll)$ for some transitive relation $\ll$.
% This represented space is a Quasi-Polish Space for all computably enumerable transitive relations $\ll$.

Now, a space $\mathbf{X}$ which is computably isomorphic to $\mathbf{I}(\ll)$ for some \emph{computably enumerable} transitive relation $\ll$ on $\N$ is called a \emph{precomputably Quasi-Polish}.
If $\mathbf{X}$ is precomputably Quasi-Polish and computably overt, then $\mathbf{X}$ is called \emph{computably Quasi-Polish}.
All precomputably Quasi-Polish spaces are computably admissible and effectively countably based.

\subsection{Computable metric spaces}

A \emph{(complete) computable metric space} is specified by a computable pseudometric $d \colon \N \times \N \to \R$ on the natural numbers.
The metric induces a represented space $\mathbf{X}_d$ as follows:
The underlying set of $\mathbf{X}_d$ is the Cauchy completion of $(\N, d)$.
A sequence $p \in \NN$ is a name of a point $x \in \mathbf{X}_d$ if and only if $d(p(n),x) < 2^{-n}$ for all $n$.
The space $\mathbf{X}_d$ is a complete metric space. The distance function $d \colon \mathbf{X}_d \times \mathbf{X}_d \to \R$ is computable.
The space $\mathbf{X}_d$ is a computably Hausdorff computably Quasi-Polish space.
More generally, we call a represented space $\mathbf{X}$ a computable metric space if it is isomorphic to $\mathbf{X}_d$ for some computable pseudometric $d$ on $\N$. A represented space is a complete computable metric space if and only if it is isomorphic to a computably overt computably $\Pi^0_2$-subset of the Hilbert cube $[0,1]^\omega$.

\subsection{Computably normal spaces}
In Section
\ref{sec:discrete_precompQP}
below we will present a discrete precomputably Quasi-Polish space which was (essentially) exhibited by Weihrauch as an example to separate ``computably regular'' from ``computably metrizable''.
We now introduce computably regular and computably normal spaces.

\begin{definition}
We call $\mathbf{X}$ computably regular, if the multi-valued map $\operatorname{Reg} : \subseteq \mathbf{X} \times \mathcal{A}(\mathbf{X}) \mto \mathcal{O}(\mathbf{X}) \times \mathcal{O}(\mathbf{X})$ with $(x, A) \in \dom(\operatorname{Reg})$ iff $x \notin A$ and $(U,V) \in \operatorname{Reg}(x,A)$ iff $x \in U$, $A \subseteq V$ and $U \cap V = \emptyset$ is computable.
\end{definition}

\begin{proposition}
If $\mathbf{X}$ is computably Hausdorff and computably discrete, then it is computably regular.
\end{proposition}
\begin{proof}
Given the input $(x,A)$, we ignore $A$ and return the pair $\{x\}$ and $\mathbf{X} \setminus \{x\}$. The former we obtain by computable discreteness of $\mathbf{X}$, the latter by computable Hausdorffness.
\end{proof}

\begin{definition}
We call $\mathbf{X}$ computably normal iff the map $\operatorname{Norm} : \subseteq \mathcal{A}(\mathbf{X})^2 \mto \mathcal{O}(\mathbf{X})^2$ with $(A,B) \in \dom(\operatorname{Norm})$ iff $A \cap B = \emptyset$ and $(U,V) \in \operatorname{Norm}(A,B)$ iff $U \cap V = \emptyset$ and $A \subseteq U$ and $B \subseteq V$ is computable.
\end{definition}

\begin{definition}
We call $\mathbf{X}$ computably hereditarily normal iff the map $\operatorname{HeNorm} : \mathcal{A}(\mathbf{X})^2 \mto \mathcal{O}(\mathbf{X})^2$ with $(U,V) \in \operatorname{HeNorm}(A,B)$ iff $U \cap V = \emptyset$ and $A \setminus B \subseteq U$ and $B \setminus A \subseteq V$ is computable.
\end{definition}

The following alternate characterization might illuminate why we call this notion ``computably hereditarily normal''; as it is equivalent to saying that every open subspace of a computably hereditarily normal is computably normal in a uniform way:

\begin{proposition}[\footnote{This result was suggested to us by an anonymous referee.}]
The following are equivalent for a represented space $\mathbf{X}$:
\begin{enumerate}
\item $\mathbf{X}$ is computably hereditarily normal.
\item The map $\operatorname{NormSub} :\subseteq \mathcal{O}(\mathbf{X}) \times \mathcal{A}(\mathbf{X}) \times \mathcal{A}(\mathbf{X}) \mto \mathcal{O}(\mathbf{X})^2$ is computable, where $(Y,A,B) \in \dom(\operatorname{NormSub})$ iff $Y \cap A \cap B = \emptyset$ and $(U,V) \in \operatorname{NormSub}(Y,A,B)$ if $(Y \cap A) \subseteq U$ and $(Y \cap B) \subseteq V$ and $Y \cap U \cap V = \emptyset$.
\end{enumerate}
\begin{proof}
To show that $(1) \Rightarrow (2)$, we observe that if $(Y,A,B)$ in an instance of $\operatorname{NormSub}$, then any solution $(U,V) \in \operatorname{HeNorm}(A,B)$ already satisfies $(U,V) \in \operatorname{NormSub}(Y,A,B)$.

For the other direction, if we are given some instance $(A,B)$ of $\operatorname{HeNorm}$, we can compute $Y := X \setminus (A \cap B) \in \mathcal{O}(\mathbf{X})$, and find that $(Y,A,B)$ is a valid input for $\operatorname{NormSub}$. If $(U,V) \in \operatorname{NormSub}(Y,A,B)$, then $(Y \cap U,Y \cap V) \in \operatorname{HeNorm}(A,B)$.
\end{proof}
\end{proposition}

If we want to consider more general subspaces, we need to restrict our attention to effectively countably based spaces (as otherwise subspaces may not be well-behaved).

\begin{proposition}
\label{prop:hereditarilynormal}
If $\mathbf{Y}$ is effectively countably based, computably admissible, and computably hereditarily normal, then every subspace of $\mathbf{Y}$ is computably hereditarily normal.
\end{proposition}
\begin{proof}
    Let $\mathbf{X}$ be a subspace of $\mathbf{Y}$.
It follows immediately from Proposition \ref{Proposition: extending open sets in ccb spaces} that the restriction map
$\A(\mathbf{Y}) \to \A(\mathbf{X}), A \mapsto A \cap X$
has a computable multi-valued right-inverse $s \colon \A(\mathbf{X}) \rightrightarrows \A(\mathbf{Y})$.
Letting $r \colon \O(\mathbf{Y}) \to \O(\mathbf{X}), r(U) = U \cap X$ denote the restriction map on open sets,
it is straight-forward to check that the multi-valued map
\[
  {(r \times r) \circ {\operatorname{HeNorm}} \circ {(s \times s)}} ~~ \colon ~~ {\mathcal{A}(\mathbf{X})^2 \mto \mathcal{O}(\mathbf{X})^2}
\]
witnesses that $\mathbf{X}$ is effectively hereditarily normal.
\end{proof}

\begin{proposition}
\label{prop:hilbertcubehereditarilynormal}
$[0,1]^\omega$ is effectively hereditarily normal.
\end{proposition}
\begin{proof}
Given $A,B \in \A\left([0,1]^{\omega}\right)$ we can (non-deterministically) compute an enumeration $(\widehat{V}_n)_n$ of all closed balls with rational radius and centre that are contained in $[0,1]^\omega \setminus A$, and an enumeration $(\widehat{U}_n)_n$ of all such closed balls that are contained in $[0,1]^\omega \setminus B$.

For a closed ball $\widehat{B}$, let $B$ denote the corresponding open ball.
Let
\[
    U = \bigcup_{n \in \mathbb{N}} \left(U_n \setminus \bigcup_{k < n} \widehat{V}_k\right)
\qquad
\text{and}
\qquad
V = \bigcup_{m \in \mathbb{N}} \left(V_m \setminus \bigcup_{k \leq m} \widehat{U}_k\right).
\]
Clearly, $U$ and $V$ are uniformly computable in $A$ and $B$.

We claim that $U$ and $V$ are disjoint.
Consider a set of the form $U_n \setminus \bigcup_{k < n} \widehat{V}_k$.
By construction, this set is disjoint from
$V_m \setminus \bigcup_{k \leq m} \widehat{U}_k$
for $m \geq n$.
Again by construction, it is disjoint from
$V_m$ for $m < n$, so it is a fortiori disjoint from
$V_m \setminus \bigcup_{k \leq m} \widehat{U}_k$.
This shows that $U$ and $V$ are disjoint.

Now, let $x \in A \setminus B$.
Since $x \notin B$ and $(\widehat{U}_n)_n$ enumerates \emph{all} closed rational balls that are contained in $[0,1]^\omega \setminus B$, we have $x \in U_n$ for some $n$.
Since $x \in A$ we have $x \notin \widehat{V}_k$ for all $k$.
So $x \in U_n \setminus \bigcup_{k < n} \widehat{V}_k$ and hence $x \in U$.
Thus, $U \supseteq A \setminus B$.
An analogous argument establishes $V \supseteq B \setminus A$.
\end{proof}

\section{Computably discrete computably Quasi-Polish spaces}
\label{sec:compquasipolish}
\manlio{I guess we should either say computably discrete computably qP, or just say discrete qP. }
\eike{My take on this: exempt section headings from our convention to drop ``computably'' and add comptuably discrete etc. everywhere.}

\greybox{To avoid drowning the text in occurrences of the word \emph{computably}, we adopt the convention that from this point onwards, ``discrete'', ``Hausdorff'', ``admissible'', ``overt'', ``fibre-overt'', and ``isomorphic'' all refer to the computable version, and use the modifier \emph{classical} to identify the rare cases where we do not mean the computable version.}

\manlio{Oh ok, I didn't realize that we're not using the convention for QP. Where do we mention this? At a glance, it seems that the background is written without this convention.  }
\eike{The idea is that the convention starts just after this section. But we should not extend the convention to QP because we want to distinguish computably QP from precomputably QP}
\manlio{I don't mind the box but I would have been happy with less, like a remark environment :) Should this go here or at the end of the background?}
We start our investigation by looking at discrete computably Quasi-Polish spaces. Computably Quasi-Polish spaces tend to be a setting where everything in computable topology works out very nicely. They are the computable version of Quasi-Polish spaces \cite{debrecht6} proposed in \cite{stull,paulydebrecht4}. Indeed, we can obtain several characterizations. We first start with the following proposition.

\begin{proposition}
    \label{prop:discrete_quotient}
    A quotient of $\mathbb{N}$ by an equivalence relation $R$ is admissible iff it is discrete.
\end{proposition}
\begin{proof}
    For the left-to-right implication, we first observe that, given $n, m \in \mathbb{N}$, we can compute a name for
    \[ \mathcal{F} = \{U \in \mathcal{O}(\mathbb{N}/R) \mid [n]_R \in U \land [m]_R \in U\}\in \mathcal{O}(\mathcal{O}(\mathbb{N}/R)).\]
    Let $\Phi$ be a computable functional witnessing the admissibility of $\mathbb{N}/R$. It is not hard to check that $\Phi(\mathcal{F})$ produces an output iff $[n]_R = [m]_R$. Indeed, if $x := [n]_R = [m]_R$, then $\mathcal{F}$ is the neighbourhood filter of $x$, hence $\Phi(\mathcal{F})=x$. Conversely, assume that $\Phi(\mathcal{F})$ commits to some $y$. Since any prefix of a name for $\mathcal{F}$ can be extended to a name for a neighbourhood filter of $[n]_R$ or $[m]_R$, the admissibility of $\mathbb{N}/R$ implies that $y = [n]_R = [m]_R$. This proves that $\mathbb{N}/R$ is discrete.

    For the converse direction, assume that $\mathbb{N}/R$ is discrete. In particular, for every $m\in \mathbb{N}$, we can compute $U_m:=\{ [m]_R \}\in \mathcal{O}(\mathbb{N}/R)$. Given $\mathcal{U}:=\{ U \in \mathcal{O}(\mathbb{N}/R) \mid [n]_R \in U\}\in \mathcal{O}(\mathcal{O}(\mathbb{N}/R))$, we can computably search for some $m\in \mathbb{N}$ such that $U_m\in \mathcal{U}$. Clearly, any such $m$ is such that $[m]_R=[n]_R$. Since we will eventually find one, this witnesses the admissibility of $\mathbb{N}/R$.
\end{proof}

\begin{theorem}
\label{theo:compquasipolish}
For discrete $\mathbf{X}$, the following are equivalent:
\begin{enumerate}
\item $\mathbf{X}$ admits a computably equivalent representation with domain $\Baire$.
\item There exists a computable surjection $s : \mathbb{N} \to \mathbf{X}$.
\item $\mathbf{X}$ is computably Quasi-Polish. %(which includes being admissible and effectively countably-based).
\item $\mathbf{X}$ is isomorphic to a discrete quotient of $\mathbb{N}$.
\item $\mathbf{X}$ is isomorphic to an admissible quotient of $\mathbb{N}$.
\end{enumerate}
\end{theorem}
\begin{proof}
$(1 \Rightarrow 2)$:
Let $\delta : \Baire \to \mathbf{X}$ be a total representation for $\mathbf{X}$ and let $\mathrm{isEqual} : \mathbf{X} \times \mathbf{X} \to \mathbb{S}$ be a computable map witnessing the discreteness of $(\mathbf{X},\delta)$. Using a realizer $E : \Baire \times \Baire \to \Cantor$ for $\mathrm{isEqual}$, we can compute a list $(w_i)_{i \in \mathbb{N}}$ of all $w\in\baire$ such that $E$ writes the first $1$ upon reading the pair $(w,w)$. Since $\delta$ is total, for every $p\in\Baire$ we know that $E(p,p)$ contains a $1$. In particular, the sequence $(w_i)_{i \in \mathbb{N}}$ identifies an open cover of $\Baire$. Moreover, the discreteness assumption implies that $\delta$ is constant on each cylinder $w_i\Baire$. As such, we can define a computable surjection $s : \mathbb{N} \to \mathbf{X}$ as the map realized by $i \mapsto w_i0^\omega$.

$(2 \Rightarrow 1)$: Let $\delta_\mathbb{N}$ be a total representation of $\mathbb{N}$. We claim that $\delta_{\mathbf{X}}' := s \circ \delta_\mathbb{N}$ is a total representation of $\mathbf{X}$. The translation from $\delta_{\mathbf{X}}'$ to $\delta_\mathbf{X}$ is realized by a realizer of $s$. For the converse, given a $\delta_\mathbf{X}$-name $p$ we exhaustively generate all $s(n) \in \mathbf{X}$, and use $\mathrm{isEqual}$ to identify some $n$ with $s(n) = \delta_\mathbf{X}(p)$, and pick some $q \in \delta_\mathbb{N}^{-1}(\{n\})$. Then $q$ is a $\delta'_\mathbf{X}$-name for the same point as $q$, given the translation in the other direction.

$(1 \wedge 2 \Rightarrow 3)$: We first show that $2.$ implies $\mathbf{X}$ being admissible. Assume we are given $\{U \in \mathcal{O}(\mathbf{X}) \mid x \in U\} \in \mathcal{O}(\mathcal{O}(\mathbf{X}))$. Using $s$ and discreteness of $\mathbf{X}$, we can generate all $\{s(n)\} \in \mathcal{O}(\mathbf{X})$ for $n \in \mathbb{N}$, and search for some $n$ with $\{s(n)\} \in \{U \in \mathcal{O}(\mathbf{X}) \mid x \in U\}$. Since $s$ is surjective, such an $n$ needs to exist, and it obviously satisfies that $s(n) = x$. As such, we can compute $x$, thus witnessing the admissibility of $\mathbf{X}$.

Next, we observe that the representation $\delta : \Baire \to \mathbf{X}$ from point 1 is fibre-overt. Given a prefix $w$ and a point $x \in \mathbf{X}$ we can search for extensions $p$ of $w$ such that the equality test on $\mathbf{X}$ confirms that $\delta(p) = x$. As $\Baire$ itself is computably overt, this is an effective process and yields the claim. By \cite[Theorem 14]{paulydebrecht4} a total admissible fibre-overt representation characterizes computably Quasi-Polish spaces.
\eike{Potential minor TODO: use the definition of ``effectively fibre-overt'' (or rather, that of overtness) more directly in this argument to not confuse readers who are seeing this for the first time. }

$(3 \Rightarrow 1$): By \cite[Theorem 14]{paulydebrecht4} a computably Quasi-Polish spaces admits a total representation.

$(2 \Rightarrow 4)$: Let $n \cong m$ iff $s(n) = s(m)$. We claim that $\mathbf{X}$ is isomorphic to $\mathbb{N}/\cong$. Clearly, the computable map $\varphi : \mathbb{N}/\cong \to \mathbf{X}$ induced by $s$ is a bijection. Conversely, given $x \in X$ we can use discreteness of $\mathbf{X}$ to (non-deterministically) identify some $n \in \mathbb{N}$ with $s(n) = x$. Hence, $\varphi^{-1}$ is computable, too. The discreteness of $\mathbb{N}/\cong$ follows immediately from the discreteness of $\mathbf{X}$.

$(4 \Rightarrow 2)$: Straight-forward.

$(4 \Leftrightarrow 5)$: This follows by Proposition~\ref{prop:discrete_quotient}.
\end{proof}

We can observe that there is a close connection between discrete Quasi-Polish spaces and the represented spaces of equivalence classes of {computably enumerable equivalence relations} on $\mathbb{N}$ (ceers). In particular, if $R$ is a ceer, we write $\mathbb{N}/R$ for the represented space of $R$-equivalence classes, where each class is represented via any of its representatives. As mentioned in the introduction, ceers exhibit a rich structure and they have received significant attention in recent years.

The equivalence of the points (3), (4), and (5) in the previous theorem can be restated as follows:

\begin{corollary}
    Let $\mathbf{X}$ be a represented space. The following are equivalent:
    \begin{enumerate}
        \item $\mathbf{X}$ is a discrete Quasi-Polish space.
        \item $\mathbf{X}$ there is an equivalence relation $R$ such that $\mathbb{N}/R$ is admissible and $\mathbf{X}$ is isomorphic to $\mathbb{N}/R$.
        \item $\mathbf{X}$ there is a ceer $R$ such that $\mathbf{X}$ is isomorphic to $\mathbb{N}/R$.
    \end{enumerate}
\end{corollary}

\begin{proposition}
\label{prop:compquasipolish}
For an infinite discrete computably Quasi-Polish space $\mathbf{X}$ the following are equivalent:
\begin{enumerate}
\item $\mathbf{X} \cong \mathbb{N}$.
\item There is a computable injection $\iota : \mathbf{X} \to \mathbb{N}$.
\item $\mathbf{X}$ is a computable metric space.
\item $\mathbf{X}$ is Hausdorff.
\end{enumerate}
\end{proposition}
\begin{proof}
It is immediate that $1$ implies $2$ and that $2$ implies $4$. It is also immediate that $1$ implies $3$ and that $3$ implies $4$. We only need to show that $4$ implies $1$. By Theorem \ref{theo:compquasipolish} we have a computable surjection $s : \mathbb{N} \to \mathbf{X}$. Since $\mathbf{X}$ is discrete and Hausdorff, we find that $s(n) = s(m)$ is decidable for $n,m \in \mathbb{N}$. Consequently, the set $S = \{n \in \mathbb{N} \mid \forall i < n \ s(i) \neq s(n)\}$ is a decidable infinite subset of $\mathbb{N}$, which means there is a (computable) isomorphism $\sigma : \mathbb{N} \to S$, which we can lift to yield an isomorphism between $\mathbf{X}$ and $\mathbb{N}$.
\end{proof}

In computable topology, being discrete does not imply being Hausdorff. But there is even more, and we can exhibit a discrete computably Quasi-Polish space having no decidable non-trivial properties at all. This answers a question posed to us by Emmanuel Rauzy.

\begin{example}
\label{ex:discretenonontrivialdecidable}
There is an infinite discrete computably Quasi-Polish space $\mathbf{X}$ such that every computable $f : \mathbf{X} \to \mathbb{N}$ is constant.
\end{example}
\begin{proof}
We build the space $\mathbf{X}$ from a directed graph $G$ with vertex set $\mathbb{N}$. Each vertex will have out-degree at most $1$. For the vertex $n$, we first wait for confirmation that the $n$-th Turing machine halts on $n$ and outputs some number $a_n$. If this never happens, $n$ will have out-degree $0$. If it does happen, we then search for some $m \neq n$ such that the $n$-th TM halts on $m$ and outputs some $a_m \neq a_n$. If we do find such an $m$, we add an edge $n \mapsto m$ for the first candidate found.

Let $n \equiv m$ if there is an \emph{undirected} path between $n$ and $m$ in $G$. This is a computably enumerable relation. Let $\mathbf{X} = \mathbb{N}/\equiv$. By Theorem \ref{theo:compquasipolish}, this yields a discrete computably Quasi-Polish space.

To see that $\mathbf{X}$ is infinite, we observe that there are infinitely many $n$ such that the $n$-th TM does not halt on $n$. This means that there are infinitely many vertices in $G$ with out-degree $0$, and -- since all vertices in $G$ have out-degree at most $1$ -- no two such vertices can be connected by an undirected path. This means there are infinitely many connected components in the graph that is obtained from $G$ by forgetting edge directions, and thus infinitely many points in $\mathbf{X}$.

For the final claim, assume for the sake of a contradiction that the $n$-th Turing machine computes a non-constant function $f : \mathbf{X} \to \mathbb{N}$. Because $f$ is total, the $n$-th TM must halt on $n$ and output some $a_n$. Because $f$ is not constant, there is some $m$ on which the $n$-th TM outputs some different value $a_m$, which means that there will be an edge from $n$ to one such $m$ in $G$. But that means that $n$ and $m$ denote the same point in $\mathbf{X}$, and thus the $n$-th TM doesn't actually compute a function due to the failure of extensionality.
\end{proof}

Instead of directly constructing a space to witness the claim in Example \ref{ex:discretenonontrivialdecidable}, we could instead have used the connection to ceers we established and import known results from the literature. We start with an easy observation to link topological properties to those commonly studied for ceers:
\begin{observation}
The following are equivalent for a ceer $R$:
\begin{enumerate}
\item Any distinct equivalence classes of $R$ are recursively inseparable.
\item Every computable multi-valued function $F : (\mathbb{N}/R) \mto \mathbf{2}$ has a constant choice function.
\end{enumerate}
\begin{proof}
We prove contrapositives both ways.
Recall that disjoint sets $A, B \subseteq \mathbb{N}$ are called recursively separable if there exists a decidable set $C \subseteq \mathbb{N}$ with $A
\subseteq C$ and $B \subseteq \mathbb{N} \setminus C$. If two $R$-equivalence classes $A$, $B$ are recursively separable, then the witness $C$ gives rise
to a computable realizer of the multi-valued function $F : (\mathbb{N}/R) \mto \mathbf{2}$ where $F(A) = 0$, $F(B) = 1$ and $F(X) = \{0,1\}$ for
$X \in(\mathbb{N}/R) \setminus \{A, B\}$, which clearly has no constant choice functions. Conversely, if  $F : (\mathbb{N}/R) \mto \mathbf{2}$ has no constant choice function, there must be some $A \in (\mathbb{N}/R)$ with $1 \notin F(A)$ and some $B \in (\mathbb{N}/R)$ with $0 \notin F(B)$. If $F$ has a computable realizer, then this realizer witnesses that $A$ and $B$ are recursively separable.
\end{proof}
\end{observation}

Ceers whose equivalence classes are not only recursively inseparable, but are so in an
effective or uniformly effective way have been constructed and studied previously, see e.g. \cite{almnss} and Proposition 3.13 therein
(the examples built there are far more complicated as they are meant to satisfy much more specific properties).

\section{Computably discrete precomputably Quasi-Polish spaces}
\label{sec:discrete_precompQP}
The difference between a computably Quasi-Polish space and a precomputably Quasi-Polish is that the former is required to be overt, but the latter might not be. It is straight-forward to come up with separating example:

\begin{example}[Discrete, Hausdorff, admissible, not overt]
The complement of the Halting problem $\mathbb{H}^c = \{n \in \mathbb{N} \mid \Phi_n \uparrow\}$ is an infinite Hausdorff discrete precomputably Quasi-Polish space which is not overt.
\end{example}

With some more work we can see that the equivalence of being Hausdorff and admitting a computable injection into $\mathbb{N}$ we established for discrete computably Quasi-Polish spaces does not extend to precomputably Quasi-Polish spaces:

\begin{example}
There is a discrete Hausdorff precomputably Quasi-Polish space $\mathbf{X}$ such that there is no computable injection $\iota : \mathbf{X} \to \mathbb{N}$.
\begin{proof}
Given $n \in \mathbb{N}$, we can uniformly build a discrete Hausdorff precomputably Quasi-Polish space $\mathbf{X}_n$ such that the $n$-th computable
function does not realize an injection $\iota_n : \mathbf{X}_n \to \mathbb{N}$. Then $\sum_{n \in \mathbb{N}} \mathbf{X}_n$ is the desired example. The
space $\mathbf{X}_n$ could be the singleton $\{\mathbb{N}\} \subseteq \mathcal{O}(\mathbb{N})$, or it could be $\{I,\mathbb{N} \setminus I\} \subseteq
\mathcal{O}(\mathbb{N})$ where $I$ is a finite set given via a 2-c.e.~code\footnote{Recall that a $2$-c.e.\ set is the difference of two computably
enumerable (c.e.)sets, see e.g.\ \cite[Sec.\ 3.8.4]{Soare16}. A $2$-c.e.\ set can be represented with a pair of names for c.e.\ sets.}.
We also built a witness for computable Hausdorffness. The construction proceeds in phases:

\begin{description}
\item[Phase 1] We search for some finite set $J$ such that $\Phi_n$ outputs some $m \in \mathbb{N}$ upon seeing some enumeration of $J$. As long as we have not yet found one, $\mathbf{X}_n$ is  $\{\mathbb{N}\}$ and the witness of computable Hausdorffness is trivial. If we do find such a $J$, we proceed to Phase 2.
\item[Phase 2] We state that $I := \{0,1,\ldots,\max J + 1\}$, and add to the witness of computable Hausdorffness that enumerations containing all of $\{0,1,\ldots,\max J + 1\}$ refer to different points than enumerations containing $\max J + 2$. We search for some finite set $K \subseteq \mathbb{N} \setminus I$ such that $\Phi_n$ outputs some $\ell \neq m$ upon seeing some enumeration of $K$. If we do find such a $K$, we proceed to Phase 3.
\item[Phase 3] We change our mind regarding $I$ and assert instead that $I = J \cup K$. The witness of Hausdorffness will now also declare that enumerations containing all of $J \cup K$ refer to different points than those containing $\max K + 1$.
\end{description}

This does yield a valid witness of computable Hausdorffness, as the condition added in Phase 2 is not actually met by any point if we do reach Phase 3. If we remain in Phase 1 forever, then $\Phi_n$ is undefined everywhere. If we remain in Phase 2 forever, then $\Phi_n$ is either undefined on some names in $\mathbf{X}_n$ or is constant $m$. If we reach Phase 3, $\Phi_n$ does not act extensionally on names for $I \in \mathbf{X}_n$.
\end{proof}
\end{example}

%\subsection{On an example by Weihrauch}
%\label{subsec:weihrauch}
%In this section, we investigate a separating example by Weihrauch, given as \cite[Example 5.4]{weihrauchm}.
%\manlio{Possibly mention here what is being separated with this example.}
For the remainder of this section, we investigate a family of discrete spaces $S_A$
parameterized by some $A \subseteq \mathbb{N}$. This generalizes an example of
Weihrauch which is computably normal but not computably
regular~\cite[Example 5.4]{weihrauchm}.

\begin{definition}
Given $A \subseteq \mathbb{N}$, let $S_A := \{(n,\top) \mid n \in A\} \cup \{(n,\bot) \mid n \notin A\} \subseteq \mathbb{N} \times \mathbb{S}$.
\end{definition}

Regardless of the choice of $A$, the space $S_A$ is effectively countably-based and admissible as it inherits these properties as a subspace of $\mathbb{N} \times \mathbb{S}$. The space $S_A$ is Hausdorff and discrete, as the projection $\pi_1 : S_A \to \mathbb{N}$ is a computable injection into a Hausdorff and discrete space. It follows that $S_A$ is computably regular. Beyond that, we observe the following:

\begin{proposition}
The following are equivalent:
\begin{enumerate}
\item $S_A \cong \mathbb{N}$.
\item $S_A$ is a complete computable metric space.
\item $S_A$ is computably Quasi-Polish.
\item $S_A$ is computably separable.
\item $S_A$ is overt.
\item $A$ is computably enumerable.
\end{enumerate}
\end{proposition}
\begin{proof}
The implications $1 \Rightarrow 2 \Rightarrow 3 \Rightarrow 4 \Rightarrow 5$ are all trivial. If $S_A$ is overt we can recognize $n \in A$ by asking
whether the open set accepting only $(n,\top)$ is non-empty. If $A$ is computably enumerable, the map $n \mapsto (n, [n \in A]) : \mathbb{N} \to \mathbb{N}
\times \mathbb{S}$ is computable, and together with the projection $\pi_1 : \mathbb{N} \times \mathbb{S} \to \mathbb{N}$, it witnesses that $\mathbb{N} \cong S_A$.
\end{proof}

\begin{proposition}
The following are equivalent:
\begin{enumerate}
\item $S_A$ computably embeds into $[0,1]^\omega$.
\item $S_A$ is computably normal.
\item $A$ is in the first level of the difference hierarchy over $\Sigma_1$.
\end{enumerate}
\end{proposition}
\begin{proof}
\begin{description}
\item[$1 \Rightarrow 2$] By Propositions \ref{prop:hereditarilynormal} and \ref{prop:hilbertcubehereditarilynormal}.
\item[$2 \Rightarrow 3$] We are given some $n \in \mathbb{N}$, and initially proclaim that $n \notin A$. We run the algorithm for normality of $S_A$ on the closed sets $(\{n\} \times \mathbb{S}) \cap S_A$ and $\{(n,\bot)\} \cap S_A$. If it is actually the case that $n \in A$, then the algorithm must react by making the first open set it returns accept $(n,\top)$. If this ever happens, we change our proclamation to be that $n \in A$, and we adjust the closed sets we feed to the algorithm to be $\emptyset$ and $\{(n,\bot)\} \cap S_A$.
Then let us show that the algorithm reacts by letting the second open set it returns accept $(n, \bot)$ if and only if $n \notin A$, which is enough to
prove that $A$ is a computable difference of opens. The right-to-left direction is obvious. For the left-to-right,
assume that $n \in A$ and that the code for the second open set returned by $\operatorname{Norm}$ accepts $(n, \bot)$; by continuity
of the realizer, it means it must also accept $(n, \top)$. But then it means that $\operatorname{Norm}$
returns open two sets overlapping on $(n, \top)$ for a valid input, which is a contradiction.
%If $n \notin A$ was true after all, then the algorithm must react by letting the second open set it returns accept $(n,\bot)$. However, making it accept $(n,\bot)$ also makes it accept $(n,\top)$, so if $n \in A$ were true, the two resulting open sets would now overlap and the algorithm has failed. Thus, upon seeing that the second open set accepts $(n,\bot)$ we can be assured that $n \notin A$. This places $A$ in the first level of the difference hierarchy over $\Sigma_1$.
\eike{I don't know if the above argument could be potentially made cleaner by separating the algorithm more from the correctness proof}
\item[$3 \Rightarrow 1$] We describe an embedding $\iota : S_A \to \mathbb{N} \times [0,1]$ instead. The point $(n,b)$ gets mapped somewhere into $\{n\} \times [0,1]$. We run our $2-$c.e.\ procedure for $A$. In the first phase, we believe that $n \notin A$. We provide approximations to $\iota$ that map $(n,b)$ to $(n,0)$ and approximations to $\iota^{-1}$ that map $(n,x)$ to $(n,\bot)$ for all $x \in [0,1]$. If the $2-$c.e.\ procedure for $A$ ever declares that $n \in A$, we advance to Phase 2. There is some $\varepsilon > 0$ such that declaring $\iota(n,\top) = (n,\varepsilon)$ is still compatible with the information provided so far. We do this to one name of $(n,\top)$ at a time, and also adjust $\iota^{-1}$ such that $\iota^{-1}(n,x) = (n,\top)$ for $x > \frac{\varepsilon}{2}$. We do not provide any additional information on what $\iota(n,\bot)$ would be -- in Phase 2, we believe this value does not need to be defined. If we do learn that after all, $n \notin A$, we enter the third and final phase. We stop setting $\iota(n,\top) = \varepsilon$, and instead make $\iota(n,\bot) = (n,0)$ and
$\iota^{-1}(n,0) = (n,\bot)$ true.
\end{description}
\end{proof}

\begin{proposition}
$S_A$ is precomputably Quasi-Polish iff $A$ is $\Delta^0_2$.
\end{proposition}
\begin{proof}
 As $\mathbb{N} \times \mathbb{S}$ is computably Quasi-Polish, $S_A$ is precomputably Quasi-Polish iff it is a $\Pi^0_2$-subspace of $\mathbb{N} \times \mathbb{S}$. The latter implies that $A$ is $\Delta^0_2$, since $n \in A \Leftrightarrow (n,\top) \in S_A$ and $n \notin A \Leftrightarrow (n,\bot) \in S_A$. Conversely, if $A$ is $\Delta^0_2$, then so is $S_A$.
\end{proof}

\section{Finite spaces}
\label{sec:finite_spaces}

\begin{proposition}
When $\mathbf{X}$ has cardinality $n$, is classically discrete and has a fibre-overt representation, there is a computable injection $\iota : \mathbf{X} \to \mathbf{n}$.
\end{proposition}
\begin{proof}
Classically, we can select names $p_1, \ldots, p_n$ for the $n$ points in $\mathbf{X}$. The classical discreteness ensures that these have prefixes $w_1,\ldots,w_n$ such that no $w_i$ extends to a prefix of some $p_j$ for $i \neq j$. Since the $(w_i)_{i \leq n}$ contain only finite information, we can use these as parameters for $\iota$.

Given some $x \in \mathbf{X}$, we can use the fibre-overtness of the representation to ask for each $i$ whether $x$ has a name starting with $w_i$. Since this holds for exactly one $i$, we can determine this effectively -- this process computes the desired $\iota$.
\end{proof}

\begin{corollary}
If $\mathbf{X}$ is finite, classically discrete and effectively countably-based, then it is already discrete and Hausdorff.
\end{corollary}

\begin{proposition}
If $s : \mathbf{n} \to \mathbf{X}$ is a computable surjection, then $\mathbf{X}$ is discrete iff it is Hausdorff.
\begin{proof}
W.l.o.g. we may assume that $s$ is even a bijection. Moreover, from the assumption that $\mathbf{X}$ is either discrete or Hausdorff, it will follow that any bijective computable map $s : \mathbf{n} \to \mathbf{X}$ is actually a computable isomorphism, and thus that $\mathbf{X}$ is both discrete and Hausdorff: If $\mathbf{X}$ is discrete, given $x \in \mathbf{X}$ we can test for all $i \in \mathbf{n}$ if $s(i) = x$ and thereby identify $s^{-1}(x)$. If $\mathbf{X}$ is Hausdorff, given $x \in \mathbf{X}$ we test for all $i \in \mathbf{n}$ if $s(i) \neq x$ until we confirm this for all $i \in \mathbf{n} \setminus \{j\}$, and then we know that $s(j) = x$.
\end{proof}
\end{proposition}

\subsection{Computably discrete but not computably Hausdorff two-point space}

\begin{definition}
For an infinite and co-infinite set $A \subset \mathbb{N}$, we define the space $D_A$ with underlying set $\{a,b\}$ where $p$ is a name for $a$ if $p$ is obtained from the characteristic function of $A$ by replacing finitely many $1$s with $0$s, and $p$ is a name for $b$ if it is obtained from the characteristic function of $\mathbb{N}\setminus A$ by replacing finitely many $1$s with $0$s.
\end{definition}

\begin{observation}
The space $D_A$ is discrete.
\begin{proof}
Whenever two valid names have a $1$ in the same position, they denote the same point in $D_A$. Conversely, any two names for the same point in $D_A$ share a $1$ somewhere.
\end{proof}
\end{observation}

\begin{proposition}
\label{prop:twopointnothausdorff}
There exists $A$ such that $D_A$ is not computably Hausdorff.
\end{proposition}
\begin{proof}
A realizer of the Hausdorffness of $D_A$ essentially consists of an effective enumeration of pairs $(w_i,u_i)$ such that seeing names with such prefixes leads to the claim that they refer to distinct points in $D_A$. We will diagonalize against all of them by determining $A \cap \{0,1,\ldots,s_n\}$ for some $s_n \in \mathbb{N}$ in stage $n$. If the $n$-th candidate witness does not include a pair $(w_i,u_i)$ such that $w_i$ and $u_i$ start with $0^{s_{n-1}}$, the witness fails to provide a required answer for some names. We assert $s_{n-1} + 1 \in A$, $s_{n-1} +2 \notin A$ and set $s_n = s_{n-1} + 2$. Otherwise, we pick such a pair $(w_i,u_i)$ and set $k \in A$ for every $k$ such that $w_i(k) = 1$ or $u_i(k) = 1$. For every other $k$ between $s_{n_-1}$  and $s_n := \max \{|u_i|,|w_i|\} + 2$ (exclusive of the bounds) we assert $k \notin A$. We also set $s_n \in A$. This ensures that the $n$-th witness incorrectly asserts that some names both denoting $a$ actually refer to distinct points. Since at any stage we place some numbers into $A$ and some into its complement, we do build an infinite and co-infinite set.
\end{proof}

We can adjust the construction of $D_A$ by partitioning $\mathbb{N} \setminus A$ into finitely or infinitely many sets $B_0,B_1,\ldots$ and correspondingly the point $b$ into $b_0,b_1,\ldots$ where $p$ is a name for $b_i$ if is obtained from the characteristic function of $B_i$ by replacing finitely many $1$s by $0$s. This retains discreteness. If we construct $A$ as in Proposition \ref{prop:twopointnothausdorff}, this yields a discrete countably infinite space without any non-trivial decidable properties.

\subsection{Computably Hausdorff but not computably discrete two-point space}
\label{sec:further_results}

The notion of Hausdorffness for represented spaces does not line up precisely with the topological definition. If a represented space is Hausdorff, the corresponding topological space is only sequentially Hausdorff (this difference does not appear for countably-based spaces). A stronger notion of effective Hausdorffness that also implies topological
Hausdorffness was proposed by Schr\"oder in~\cite[Definition 1]{schrodercca21}:

\begin{definition}
We say $\mathbf{X}$ admits a computable witness of Hausdorffness if there
are computable sequences $(U_i)_{i \in \mathbb{N}}, (V_i)_{i \in \mathbb{N}} \in \mathcal{O}(\mathbf{X})^\mathbb{N}$
of opens such that $\bigcup_{i \in \mathbb{N}} U_i \times V_i = \mathbf{X}^2 \setminus \Delta_{\mathbf{X}}$. %\{ (x, y) \in \mathbf{X}^2 \mid x \neq y \}$.
\end{definition}

At the Oberwolfach Meeting 2117 ``Computability Theory'' in 2021, Brattka raised the question whether a space which is (computably) Hausdorff and topologically Hausdorff will admit a computable witness of Hausdorffness. We provide a counterexample in Corollary \ref{corr:answer}.

% \keywords{MSC-class: 	03E15, 54H05, 03D60, 03F15}
For this subsection, by enumerations of sets $A \subseteq \mathbb{N}$, we mean
sequences $s : \mathbb{N} \to \mathbb{N} \uplus \{ \bot \}$ such that
$A = \{ s_n \mid n \in \mathbb{N}, s_n \neq \bot\}$.

\begin{definition}
Fix some non-empty $A \subsetneq \mathbb{N}$.
We define the two-point space $H_A$  with underlying set $\{a,b\}$ as follows:
A name for $a$ starts with some $n \notin A$ followed by an enumeration of $A$.
A name for $b$ starts with some $n \in A$ followed by an enumeration of some
$B \subseteq \mathbb{N} \setminus A$, possibly empty.
\end{definition}

\begin{proposition}
$H_A$ is (computably) Hausdorff.
\end{proposition}
\begin{proof}
Assume we are given two names $\langle n, U\rangle$ and $\langle m, V\rangle$. If we find that $n \in V$ or $m \in V$ holds, then these have to be names for distinct points in $H_A$. Conversely, if they are names for distinct elements, then one of them is a name for $a$, and its corresponding enumerated set will contain the index of the other. As $n \in V$ or $m \in V$ is recognizable, the claim follows.
\end{proof}

Here $H_A$ is also topologically Haudorff as it is easily seen to be classically
homeomorphic to $\mathbf{2}$.

\begin{proposition}
The sets of names for $b \in H_A$ and enumerations of $A$ are Medvedev-equivalent.
%$\{1\}^{\mathcal{O}(H_A)} \equiv_{\mathrm{T}} A^{\mathcal{O}(\mathbb{N})}$ \quad (\footnote{This is using the notation from \cite{pauly-kihara3}, it is stating that names for the open set $\{1\}$ in $\mathcal{O}(H_A)$ uniformly compute enumerations of $A$ and vice versa.})
\end{proposition}
\begin{proof}
Clearly, if we have an enumeration of $A$ we build a name for $\{b\} \subseteq H_A$
by accepting $(n,V)$ whenever we notice $n \in A$.
Conversely, any realizer for $\{b\} \in \mathcal{O}(H_A)$ must accept
$(n,\emptyset)$ if and only if $n \in A$; this is because a prefix of the enumeration of
$\emptyset$ can always be extended to an enumeration of $A$, and thus a valid name of $a$
in case that $n \notin A$.
\end{proof}

\begin{corollary}
\label{corr:answer}
For non-c.e.\ $A \subseteq \mathbb{N}$, the space $H_A$ is topologically Hausdorff and (computably) Hausdorff but admits no witness of computable Hausdorffness.
\end{corollary}
\begin{proof}
From the definition, it is clear that a finite space admits a computable witness
of Hausdorffness if and only if every singleton is computably open.
But $\{b\} \in \mathcal{O}(H_A)$ is not computably open if $A \subseteq \mathbb{N}$ is not computably enumerable.
\end{proof}

\begin{corollary}
For non-c.e.\ $A \subseteq \mathbb{N}$ the space $H_A$ is finite, (computably) Hausdorff but not (computably) discrete.
\end{corollary}
\begin{proof}
In a discrete space $\mathbf{X}$, the singleton $\{x\}$ needs to be computably open for every computable point $x \in \mathbf{X}$. The point $b
\in H_A$ is computable, but here $\{b\}$ is not computably open. Thus, $H_A$ cannot be computably discrete.
\end{proof}

%I do not know whether $H_A$ is computably admissible or whether it is computably overt.

\begin{corollary}
If $A$ is not c.e., then $H_A$ cannot be both overt and admissible.
\end{corollary}
\begin{proof}
Note that a Hausdorff finite space can only fail to be
discrete if some of its points are non-computable (in this case, this is
because computing $\{a\} \in H_A$ is as hard as enumerating $A$).
In an admissible and overt space, every computably open
singleton contains a computable point.
\end{proof}

We do not know whether $H_A$ is overt, admissible or neither in general.

\begin{proposition}
If $H_A$ is overt, then $A$ is cototal.
\end{proposition}
\begin{proof}
We need to explain how to obtain an enumeration of $A$ from an enumeration $e$ of $\mathbb{N} \setminus A$.
For each $k \in \mathbb{N}$, we can obtain a name for an open set $U_k$ by accepting a pair $(n, V)$ once we have confirmed $n \notin A$ (thanks to $e$)
and $k \in V$. If $k \in A$, then $V_k = \{a\}$; whereas if $k \notin A$, then $V_k = \emptyset$.
Thus, overtness of $H_A$ will yield an enumeration of $A$.
\end{proof}

\section{Characterizing $\mathbb{N}$ up to computable isomorphism}
\label{sec:N_isomorphism}

It would be very pleasing to have a simple characterization of $\mathbb{N}$ up to isomorphism in a general computable topology setting.
Proposition \ref{prop:compquasipolish} tells us that $\mathbb{N}$ is, up to isomorphism, the only discrete Hausdorff non-compact computably Quasi-Polish space.
We raise the question whether this result can be improved.

\begin{question}
Is $\mathbb{N}$ the only represented space (up to isomorphism) which is discrete, Hausdorff, overt and admissible but not compact?
\end{question}

We have already presented examples showing that none of the criteria in our question can be dropped, with the exception of admissibility. This is covered by the following:

\begin{example}[Discrete, overt, Hausdorff, no onto surjection, not admissible]
  Pick a non-computable $p \in \mathbb{N}^\mathbb{N}$. Let $\delta_{p\mathbb{N}}(0^n1q) = n$ if $p \equiv_T q$, and $q \notin \dom(\delta_{p\mathbb{N}})$ for $q \not\equiv_T p$. The resulting space $p\mathbb{N}$ is discrete, Hausdorff and overt. It is not admissible \eike{TODO? Explain why? Neighbourhood filter of $[1p]$ is computable but no representative is computable?}, and there is no computable surjection $s : \mathbb{N} \to p\mathbb{N}$ (in fact, any function from $\mathbb{N}$ to $p\mathbb{N}$ needs to compute $p$).
\end{example}

To round off the discussion, let us point out that there are even more ways to construct spaces which are classically homeomorphic to $\mathbb{N}$, yet behave very differently in computable topology.
A last example, which is overt but not Hausdorff nor discrete is the following.

\begin{example}
There exists a space $\mathbb{N}'$ with underlying set $\mathbb{N}$ such that $\id : \mathbb{N} \to \mathbb{N}'$ is computable, and $\id : \mathbb{N}' \to \mathbb{N}$ is computable
relative to $\emptyset'$ such that $\emptyset$ and $\mathbb{N}$ are the only computable elements in $\mathcal{O}(\mathbb{N}')$.
\begin{proof}
Let $\mathrm{BB} : \mathbb{N} \to \mathbb{N}$ be the busy beaver function. We define $\delta : \Baire \to \mathbb{N}$ by $\delta(p) =
p(\mathrm{BB}(p(0)))$, and let $\mathbb{N}' = (\mathbb{N},\delta)$. Any constant sequence $n^\omega$ is name for $n$, and having access to $\emptyset'$
lets us compute $\mathrm{BB}$ and thereby decode $\delta$-names. Now let $U \in \mathcal{O}(\mathbb{N}')$ be a non-empty open set, and let $n \in U$. A
computable realizer for $U$ needs to accept any sequence $mn^\omega$. If it accepts any such sequence before having read a prefix of length
$\mathrm{BB}(m)$, then it will accept names for all numbers, i.e.~$U = \mathbb{N}$. But that means if $U \neq \mathbb{N}$, then by counting when a realizer
accepts $mn^\omega$, we obtain an upper bound for $\mathrm{BB}(m)$. This means that any realizer for non-empty $U \neq \mathbb{N}$ computes $\emptyset'$.
\end{proof}
\end{example}

\section*{Acknowledgements}
The authors thank Emmanuel Rauzy and Zhifeng Ye for fruitful discussions.

\bibliographystyle{mbibstyle}
\bibliography{references}

\end{document}